\documentclass{amsart}
\usepackage{amssymb,mathtools}
\usepackage[british]{babel}
\usepackage{enumitem}
\usepackage[latin1]{inputenc}
\usepackage{bussproof}
\usepackage{url}
\usepackage{hyperref}

\newtheorem{theorem}{Theorem}
\numberwithin{theorem}{section}
\newtheorem{corollary}[theorem]{Corollary}
\newtheorem{lemma}[theorem]{Lemma}
\newtheorem{proposition}[theorem]{Proposition}
\theoremstyle{definition}
\newtheorem{definition}[theorem]{Definition}

\newcommand{\drel}{\operatorname{\Delta-Rel}}
\newcommand{\wf}{\operatorname{WF}}
\newcommand{\ind}{\operatorname{Ind}}

\begin{document}

\title[Set-theoretic reflection and class-induction]{Set-theoretic reflection is equivalent to induction over well-founded classes}
\author{Anton Freund}
\address{Fachbereich Mathematik, Technische Universit\"at Darmstadt, Schlossgartenstr.~7, 64289~Darmstadt, Germany}
\email{freund@mathematik.tu-darmstadt.de}

\begin{abstract}
We show that induction over $\Delta(\mathbb R)$-definable well-founded classes is equivalent to the reflection principle which asserts that any true formula of first order set theory with real parameters holds in some transitive set. The equivalence is proved in primitive recursive set theory (which is weaker than Kripke-Platek set theory) extended by the axiom of dependent choice.
\end{abstract}


\maketitle

{\let\thefootnote\relax\footnotetext{\copyright~2020 American Mathematical Society.\\
First published in \emph{Proceedings of the American Mathematical Society} 148 (2020) 4503-4515, published by American Mathematical Society.}

\section{Introduction}

The present paper connects two major themes of set theory: reflection and the absoluteness of well-foundedness, where the latter is embodied by the principle of induction over well-founded classes.

Reflection principles in set theory express that any suitable property of the set-theoretic universe is already satisfied in some set. They are important for the foundations of set theory, because they have a particularly strong intrinsic justification. This is expressed in the following statement, which H.~Wang~\cite[Section~8.7]{wang96} attributes to K.~G\"odel:
\begin{equation*}
\parbox{0.9\textwidth}{``The universe of sets cannot be uniquely characterized (i.e.,~distinguished from all its initial segments) by any internal structural property of the membership relation in it which is expressible in any logic of finite or transfinite type, including infinitary logics of any cardinal number."}
\end{equation*}
As the last part of the quoted sentence suggests, reflection principles have been considered for rather general classes of properties. Often, the aim was to justify strong axioms that go beyond those of Zermelo-Fraenkel set theory. Early results in this direction are due to A.~L\'evy~\cite{levy-reflection} and P.~Bernays~\cite{bernays-reflection}. As examples of recent investigations into strong reflection principles, we cite the work of W.~Tait~\cite{tait-reflection}, P.~Koellner~\cite{koellner-reflection} and P.~Welch~\cite{welch-reflection}.

In the present paper, we focus on reflection for first order formulas in the usual language of set theory. For reasons explained below, we must also assume that all parameters are reals (i.\,e.~subsets of~$\omega$). Hence our reflection principle is given by the schema
\begin{equation*}
 \forall_{r\subseteq\omega}(\psi(r)\to\exists_N(\text{``$N$ is a transitive set''}\land\psi(r)^N)),
\end{equation*}
where $\psi\equiv\psi(x)$ is a formula with a single free variable and $\psi(r)^N$ refers to the usual notion of relativization. Note that several parameters are readily coded into a single one. It is well-known that the given reflection principle is provable in Zermelo-Fraenkel set theory, where it plays an important technical role (e.\,g.~for the justification of forcing in terms of countable transitive models). First order reflection is also important to understand the foundations of weaker set theories. In particular, the extension of Kripke-Platek set theory by the given reflection principle is one of the strongest axiom systems for which we have an ordinal analysis, which is due to~M.~Rathjen~\cite{rathjen-reflection}. In the context of this analysis, Rathjen mentions that reflection for formulas of first order set theory corresponds to $\beta$-model reflection in second order arithmetic (a detailed proof of a similar result can be found in~\cite{rathjen-mahlo-so}). Note that our restriction to real parameters is completely natural when one is interested in consequences for second order arithmetic.

Since our reflection principle is provable in Zermelo-Fraenkel set theory, we will work over a weaker base theory, namely primitive recursive set theory with infinity. This theory is based on the notion of primitive recursive set function, which has been identified by R.~Jensen and C.~Karp~\cite{jensen-karp}. Whenever we speak of primitive recursive set functions, we will assume that $\omega$ is admitted as a parameter. Primitive recursive set theory, as described by Rathjen~\cite{rathjen-set-functions}, consists of basic axioms (extensionality, regularity and infinity) and axioms which ensure that the primitive recursive set functions are total and satisfy their defining equations. Alternatively, one can read the entire paper with Kripke-Platek set theory as base theory (again with infinity; see~\cite{barwise-admissible} for an extensive introduction). The latter is stronger than primitive recursive set theory, as $\Sigma$-recursion covers primitive recursion. Note that the choice of Kripke-Platek set theory ties in with Rathjen's aforementioned ordinal analysis. Furthermore, the notion of $\Delta$-class is particularly natural in Kripke-Platek set theory, where $\Delta$-separation is available. In any case, we will eventually extend our base theory by the axiom of dependent choice (\textbf{DC}), in order to obtain a descending sequence in an ill-founded order.

We now specify the induction principle to which reflection is supposed to be equivalent. Let us agree that a relation $<_X$ on a class $X$ (both definable) is well-founded if we have
\begin{equation*}
\wf[X]:\equiv\forall_w(w\neq\emptyset\to\exists_{x\in w}\forall_{y\in w}\neg \, y<_X x).
\end{equation*}
For a formula $\varphi(x,\vec z)$ with a distinguished induction variable $x$, induction along $X$ can be expressed as
\begin{equation*}
\ind[X,\varphi]:\equiv\forall_{\vec z}(\forall_{x\in X}(\forall_{y<_X x}\varphi(y,\vec z)\to\varphi(x,\vec z))\to\forall_{x\in X}\varphi(x,\vec z)).
\end{equation*}
Note that there is no restriction on the values of the parameters~$\vec z$. By induction for $\Delta(\mathbb R)$-definable well-founded classes we mean the schema
\begin{equation*}
\wf[X]\to\ind[X,\varphi],
\end{equation*}
where $X$ is $\Delta$-definable with a real parameter. Let us stress the fact that~$X$ is not required to be set-like.

To make the statement of our induction principle more precise, we recall that a quantifier is bounded if it occurs in the form $\forall_{x\in y}\cdots$ or in the form $\exists_{x\in y}\cdots$, where $y$ is a set (not a proper class). A $\Sigma$-formula ($\Pi$-formula) is a formula in negation normal form in which all universal (existential) quantifiers are bounded. Over Kripke-Platek set theory, any $\Sigma$-formula is equivalent to a $\Sigma_1$-formula, but over primitive recursive set theory the former notion is more liberal. A $\Delta$-class is one that can be defined both by a $\Sigma$-formula and by a $\Pi$-formula. To ensure that an ordered class $(X,<_X)$ is $\Delta$-definable, it is most convenient to exhibit a $\Sigma$-formula $\theta_\Sigma(x,y,r)$ and a $\Pi$-formula $\theta_\Pi(x,y,r)$ that define the relation $x\leq_X y$. The point is that $X$ and $<_X$ can be recovered from $\leq_X$, provided that $<_X$ is well-founded and hence irreflexive. Formally, we abbreviate
\begin{equation*}
x\leq_X^r y:\equiv\theta_\Sigma(x,y,r),\quad x\in X^r:\equiv x\leq_X^rx,\quad x<_X^ry:\equiv x\leq_X^ry\land x\neq y.
\end{equation*}
To express that $\theta_\Sigma$ and $\theta_\Pi$ provide a $\Delta$-definition of $X^r$ we use the formula
\begin{equation*}
\drel[X^r]:\equiv\forall_{x,y}(x\leq_X^r y\to x\in X^r\land y\in X^r)\land\forall_{x,y}(\theta_\Sigma(x,y,r)\leftrightarrow\theta_\Pi(x,y,r)),
\end{equation*}
where the first conjunct ensures that $\leq_X^r$ does indeed have field~$X^r$. Induction over $\Delta(\mathbb R)$-definable well-founded classes can now officially be given as the schema
\begin{equation*}
\forall_{r\subseteq\omega}(\drel[X^r]\land\wf[X^r]\to\ind[X^r,\varphi]).
\end{equation*}
The proof that reflection implies induction is rather straightforward, so we give it right away. The result is, of course, a schema: Given arbitrary formulas $\theta_\Sigma$, $\theta_\Pi$ and $\varphi$ as in the induction principle, we will construct a formula $\psi$ and a proof that reflection for $\psi$ implies induction for $\varphi$ along the order defined by $\theta_\Sigma$ and $\theta_\Pi$.

\begin{proposition}\label{prop:reflection-to-induction}
Reflection for first order formulas with real parameters implies induction over $\Delta(\mathbb R)$-definable well-founded classes.
\end{proposition}
\begin{proof}
 Given formulas $\theta_\Sigma$, $\theta_\Pi$ and $\varphi$ as in the above exposition of the induction principle, we put
 \begin{equation*}
  \psi(r):\equiv\forall_{x,y}(\theta_\Sigma(x,y,r)\leftrightarrow\theta_\Pi(x,y,r))\land\neg\ind[X^r,\varphi].
 \end{equation*}
 Working in primitive recursive set theory, we fix a value $r\subseteq\omega$ of the parameter and assume that reflection for $\psi(r)$ holds. The induction principle is established by contradiction: We assume that the premises $\drel[X^r]$ and $\wf[X^r]$ hold while the conclusion $\ind[X^r,\varphi]$ fails. Then $\psi(r)$ holds, and reflection yields $\psi(r)^N$ for some transitive set~$N$. Due to the first conjunct of $\psi(r)^N$, we can use the absoluteness properties of $\Sigma$-~and $\Pi$-formulas to show that $x\in X^r$ and $x<_X^ry$ are absolute between $N$ and the universe (for $x,y\in N$). Writing $\varphi_N(x,\vec z):\equiv x\in N\to\varphi(x,\vec z)^N$, the second conjunct of $\psi(r)^N$ amounts to
 \begin{equation*}
  \neg\ind[X^r,\varphi]^N\equiv\exists_{\vec z\in N}(\forall_{x\in X^r}(\forall_{y<_X^rx}\varphi_N(y,\vec z)\to\varphi_N(x,\vec z))\land\exists_{x\in X^r}\neg\varphi_N(x,\vec z)).
 \end{equation*}
 Fix witnesses $\vec z$ for this statement. Since $\Delta_0$-separation is available in primitive recursive set theory (see e.\,g.~\cite[Corollary~1.1.10]{freund-thesis}), we can form the set
 \begin{equation*}
  w:=\{x\in N\,|\,(x\in X^r)^N\land\neg\varphi(x,\vec z)^N\}=\{x\in X^r\,|\,\neg\varphi_N(x,\vec z)\}.
 \end{equation*}
 The second conjunct of $\neg\ind[X^r,\varphi]^N$ tells us that $w$ is non-empty. By the assumption $\wf[X^r]$ we get an element $x\in w$ such that $y<_X^r x$ fails for all $y\in w$. The latter means that $\forall_{y<_X^r x}\varphi_N(y,\vec z)$ holds. Now the first conjunct of $\neg\ind[X^r,\varphi]^N$ yields~$\varphi_N(x,\vec z)$, which is incompatible with $x\in w$.
\end{proof}

As the main result of the present paper, we will prove the converse direction: induction over $\Delta(\mathbb R)$-definable well-founded classes implies reflection for first order formulas with real parameters, assuming the axiom of dependent choice. An official statement of the resulting equivalence is given as Theorem~\ref{thm:main-result} below.

Let us sketch how our main result is proved: The basic idea is to approach reflection via completeness. Assuming that $\psi(r)$ fails in every transitive model, we will construct a class-sized proof tree of~$\neg\psi(r)$ (somewhat similar to a $\beta$-proof in the sense of J.-Y.~Girard~\cite{girard-intro}). Induction over this proof will show that $\neg\psi(r)$ holds in the set-theoretic universe, so that reflection holds because its premise fails.

Our approach to completeness relies on K.~Sch\"utte's~\cite{schuette56,schuette77} method of deduction chains: The idea is to build an attempted proof of a given formula $\varphi$ from the bottom~up. This will yield either a well-founded proof of~$\varphi$ or an attempted proof with an infinite branch. From the latter one can read off a countermodel to~$\varphi$, provided that the proof search was sufficiently systematic.

For each set~$M$, we will use the method of deduction chains to search for an extensional model $(M_f,\in)\vDash\psi(r)$ with $\omega\cup\{r\}\subseteq M_f\subseteq M$. Since Sch\"utte's approach yields a countermodel to the end-formula of the attempted proof, we should search for a proof $S^r_\psi(M)$ of the formula
\begin{equation}\label{eq:end-formula}
 \neg\psi(r)\lor\neg\forall_{x,y}(\forall_z(z\in x\leftrightarrow z\in y)\to x=y),
\end{equation}
where the second disjunct ensures that the resulting model is extensional. In order to obtain a model $M_f\subseteq M$, we will work in $M$-logic. This means that $S^r_\psi(M)$ may use the infinitary rule
\vspace{.3\baselineskip}
\begin{prooftree}
\Axiom$\cdots\qquad\theta(a\fCenter)\qquad\cdots\qquad(a\in M)$
\RightLabel{,}
\UnaryInf$\forall_x\,\theta(\fCenter x)$
\end{prooftree}
\vspace{.3\baselineskip}
which allows to infer a universal formula $\forall_x\,\theta(x)$ once we have proved the premise $\theta(a)$ for each $a\in M$. The proof $S^r_\psi(M)$ can thus be realized as a labelled subtree of~$M^{<\omega}$, the tree of finite sequences with entries in $M$ (where we assume $\{0,1\}\subseteq M$ to have names for the premises of the usual unary and binary rules). The method of deduction chains will allow us to build the proofs $S^r_\psi(M)$ in a particularly uniform way. This construction can have two outcomes:
\begin{itemize}
\item There is an $M\supseteq\omega\cup\{r\}$ such that the tree $S^r_\psi(M)\subseteq M^{<\omega}$ has an infinite branch $f:\omega\to M$ (i.\,e.~we have $\langle f(0),\dots,f(n-1)\rangle\in S^r_\psi(M)$ for all $n\in\omega$).
\item The proof tree $S^r_\psi(M)$ is well-founded (with respect to end-extensions of sequences) for every set $M\supseteq\omega\cup\{r\}$.
\end{itemize}
In the first case, we will see that
\begin{equation*}
 M_f:=\{f(i)\,|\,i\in\omega\}\cup\omega\cup\{r\}\subseteq M
\end{equation*}
is an extensional model of $\psi(r)$. An application of Mostowski collapsing will then yield the transitive model required for reflection. Now consider the second case, in which all proof trees $S^r_\psi(M)$ are well-founded. Due to the uniformity of the construction, we will be able to glue the trees $S^r_\psi(M)$ into a class-sized proof tree $S^r_\psi(\mathbb V)\subseteq\mathbb V^{<\omega}$, where $\mathbb V$ denotes the universe of sets. Invoking the principle of induction over $\Delta(\mathbb R)$-definable well-founded classes, we will show that all formulas in the proof $S^r_\psi(\mathbb V)$ are true. In particular, this applies to the end-formula~(\ref{eq:end-formula}) of our proofs. Since extensionality holds in the set-theoretic universe, it follows that $\psi(r)$ must fail, so that reflection holds. Full details of the construction will be worked out in the following sections.

Based on the preceeding proof sketch, we can now explain the restriction to real parameters: Each node of a tree $S^r_\psi(M)$ does only contain a finite amount of information, which means that the information collected along a branch is at most countable. To be somewhat more precise, the method of deduction chains relies on the fact that each formula on a branch $f$ is false in the corresponding model~$M_f$. This property is proved by induction over the length of formulas. In order to ensure that a formula $\exists_x\,\theta(x)$ on $f$ fails in $M_f$, every instance $\theta(a)$ with $a\in M_f$ must thus appear on $f$ as well. However, a branch does only contain countably many formulas. Instead of reals, one could consider all hereditarily countable sets as parameters. This does not seem to increase generality, since hereditarily countable sets can be represented by well-founded trees (see e.\,g.~\cite[Section~VII.3]{simpson09}).

To explain the context of our result, we recall that equivalences between reflection and induction principles are well-known in the context of first and second order arithmetic: Due to G.~Kreisel and A.~L\'evy~\cite{kreisel68}, reflection (with parameters) over elementary arithmetic is equivalent to induction along the natural numbers, while reflection over Peano arithmetic is equivalent to induction along (the usual notation system for) the ordinal $\varepsilon_0=\min\{\alpha\,|\,\omega^\alpha=\alpha\}$ (these results have been refined by D.~Leivant~\cite{leivant83} and H.~Ono~\cite{ono87}). We point out that the previous reflection principles are usually formulated in terms of provability rather than the existence of models. However, the two formulations are equivalent in the presence of completeness. As shown by H.~Friedman~\cite{friedman-rm}, the principle of $\omega$-model reflection in second order arithmetic is equivalent to bar induction, i.\,e.~induction along arbitrary well-orders on~$\mathbb N$ (this result has been refined by S.~Simpson~\cite{simpson-bar-induction} as well as G.~J\"ager and T.~Strahm~\cite{jaeger-strahm-bi-reflection}).

Our use of deduction chains is inspired by the aforementioned paper by J\"ager and Strahm. However, one important innovation is necessary in the context of set theory: J\"ager and Strahm search for an $\omega$-model, which means that the underlying set of the model is essentially fixed in advance (second order variables are treated as predicates). In set theory, we must consider all possible supersets $M$, and the underlying set $M_f\subseteq M$ of the resulting model is determined by the branch~$f$, as described above. Our definitions of $S^r_\psi(M)$ and $M_f$ are very close to a draft~\cite{freund-bh-preprint} by the present author. In~\cite{freund-equivalence} it has been shown that the construction becomes even more uniform in the case of the constructible hierarchy: essentially, the map $\alpha\mapsto S^r_\psi(\mathbb L_\alpha)$ can be turned into a dilator in the sense of Girard~\cite{girard-pi2}. This allows for an even finer analysis of reflection under the axiom of constructibility, as shown in a first version of the present paper (available as arXiv:1909.00677v1).

To conclude this introduction, we mention two refinements of our result: Firstly, the proof of our equivalence reveals some relations between the complexity of the reflection formulas that are needed to derive induction for formulas from a given complexity class, and vice~versa. These will be discussed at the end of the paper. Secondly, one can observe that the reflection formula and the corresponding well-founded class depend on the same real parameter. In particular, it follows that reflection for first order sentences is equivalent to induction along well-founded classes that are $\Delta$-definable without parameters.

\section{Transitive models via deduction chains}

In the introduction we have described how Sch\"utte's method of deduction chains can be used to search for transitive models of set-theoretic formulas. The details of this approach will be worked out in the present section.

We fix some notation and terminology: Let $x^{<\omega}$ be the set of finite sequences with entries from a given set~$x$. By $\sigma\vartriangleleft\tau$ we express that the sequence $\sigma$ is a proper end-extension of the sequence~$\tau$. Writing $\sigma=\langle\sigma_0,\dots,\sigma_{n-1}\rangle$ and $\tau=\langle\tau_0,\dots,\tau_{m-1}\rangle$, this means that we have $m<n$ and $\sigma_i=\tau_i$ for all~$i<m$. A non-empty subset $T\subseteq x^{<\omega}$ is a tree if $\sigma\in T$ and $\sigma\vartriangleleft\tau$ imply~$\tau\in T$. In particular, any tree contains the empty sequence~$\langle\rangle$. The extension of a sequence $\sigma=\langle\sigma_0,\dots,\sigma_{n-1}\rangle\in x^{<\omega}$ by an element $a\in x$ will be denoted by $\sigma^\frown a:=\langle\sigma_0,\dots\sigma_{n-1},a\rangle$.

All object formulas that we consider will be formulas in the usual language of first order set theory (with equality). They may contain arbitrary sets as parameters (i.\,e.~each set $a$ may be used as a constant symbol with canonical interpretation~$a$). By an $M$-formula we shall mean a formula with parameters from~$M$. For technical reasons, all object formulas are assumed to be in negation normal form. This means that formulas are built from literals (negated and unnegated prime formulas) by the connectives $\land,\lor$ and the quantifiers $\forall,\exists$. To negate a formula one applies de Morgan's rules and omits double negations in front of prime formulas. Expressions such as $\neg\varphi$ and $\varphi\to\psi$ should thus be read as abbreviations (e.\,g.~one may write $\neg(x\in y\to x\in z)$ to refer to the object formula $x\in y\land\neg x\in z$). Unless noted otherwise, we assume that object formulas are closed (i.\,e.~contain no free variables).

An $M$-sequent is a finite sequence of (closed) $M$-formulas. The intuitive interpretation of a sequent $\Gamma=\langle\varphi_0,\dots,\varphi_{n-1}\rangle$ is the disjunction $\varphi_0\lor\dots\lor\varphi_{n-1}$. In particular, the order of the formulas in a sequent is irrelevant from a semantic perspective. It will, however, be crucial for the construction below. As usual, we write $\Gamma,\varphi$ rather than $\Gamma^\frown\varphi$ in the context of sequents. Similarly, we write $\varphi,\Gamma:=\langle\varphi,\varphi_0,\dots,\varphi_{n-1}\rangle$ for $\Gamma$ as before. We can now give a precise definition of the search trees from the introduction. To understand the details of the definition one should consider the proof of Proposition~\ref{prop:branch-model}.

\begin{definition}\label{def:search-tree}
Consider a formula $\psi\equiv\psi(x)$ with a single free variable $x$ and no parameters. Given a real $r\subseteq\omega$ and a set $M\supseteq\omega\cup\{r\}$, we define a tree $S^r_\psi(M)\subseteq M^{<\omega}$ and a labelling function $l_M:S^r_\psi(M)\to\text{``$M$-sequents"}$ by recursion over the finite sequences in $M^{<\omega}$. In the base of the recursion we stipulate
\begin{equation*}
\langle\rangle\in S^r_\psi(M)\qquad\text{and}\qquad l_M(\langle\rangle)=\langle\neg\psi(r),\neg\forall_{x,y}(\forall_z(z\in x\leftrightarrow z\in y)\to x=y)\rangle.
\end{equation*}
As $S^r_\psi(M)$ is to become a tree, the recursion step is only interesting for $\sigma\in S^r_\psi(M)$. We distinguish cases according to the first formula of the sequent $l_M(\sigma)=\varphi,\Gamma$ (which will never be empty): If $\varphi$ is a true literal, then $\sigma$ is a leaf of $S^r_\psi(M)$. If $\varphi$ is a false literal, then we stipulate
\begin{equation*}
\sigma^\frown a\in S^r_\psi(M)\,\Leftrightarrow\, a=0\qquad\text{and}\qquad l_M(\sigma^\frown 0)=\Gamma,\varphi,
\end{equation*}
observing $0\in\omega\subseteq M$. If $\varphi\equiv\varphi_0\land\varphi_1$ is a conjunction, then we set
\begin{equation*}
\sigma^\frown a\in S^r_\psi(M)\,\Leftrightarrow\, a\in\{0,1\}\qquad\text{and}\qquad l_M(\sigma^\frown a)=\Gamma,\varphi,\varphi_a.
\end{equation*}
If $\varphi\equiv\varphi_0\lor\varphi_1$ is a disjunction, then we define
\begin{equation*}
\sigma^\frown a\in S^r_\psi(M)\,\Leftrightarrow\, a=0\qquad\text{and}\qquad l_M(\sigma^\frown 0)=\Gamma,\varphi,\varphi_0,\varphi_1.
\end{equation*}
If $\varphi\equiv\forall_x\,\theta(x)$ is universal, then we put
\begin{equation*}
\sigma^\frown a\in S^r_\psi(M)\,\Leftrightarrow\,a\in M\qquad\text{and}\qquad l_M(\sigma^\frown a)=\Gamma,\varphi,\theta(a).
\end{equation*}
To avoid confusion, we note that $a\in M$ is automatic for $\sigma^\frown a\in S^r_\psi(M)\subseteq M^{<\omega}$. Finally, assume that $\varphi\equiv\exists_x\,\theta(x)$ is existential. Writing $\sigma=\langle\sigma_0,\dots,\sigma_{n-1}\rangle$, let $b$ be the first entry of the list
\begin{equation*}
r,\sigma_0,0,\sigma_1,1,\dots,\sigma_{n-1},n-1,n,n+1,n+2,\dots
\end{equation*}
such that $\theta(b)$ does not already occur in $\Gamma$ (note that $\omega\cup\{r\}\subseteq M$ and $\sigma\in M^{<\omega}$ ensure $b\in M$). Then set
\begin{equation*}
\sigma^\frown a\in S^r_\psi(M)\,\Leftrightarrow\, a=0\qquad\text{and}\qquad l_M(\sigma^\frown 0)=\Gamma,\varphi,\theta(b)
\end{equation*}
to complete the recursive definition of $S^r_\psi(M)$ and $l_M$.
\end{definition}

To see how the previous definition can be formalized in primitive recursive set theory, we recall some observations from~\cite[Chapter~1]{freund-thesis} (all theorem numbers in the present paragraph refer to this reference): Proposition~1.2.8 tells us that $M\mapsto M^{<\omega}$ is a primitive recursive set function (with parameter $\omega$). In particular, $M^{<\omega}$~exists as a set. Using Corollary~1.2.11, one can verify that $\sigma\in S^r_\psi(M)$ is a primitive recursive relation in $\sigma,r$ and $M$, and that $(M,\sigma)\mapsto l_M(\sigma)$ is a primitive recursive set function. Combining these facts by Corollary~1.1.10, it follows that $S^r_\psi(M)\subseteq M^{<\omega}$ is a set and that $(r,M)\mapsto S^r_\psi(M)$ is primitive recursive. Similarly, Proposition~1.2.2 tells us that the class-sized function $(M,\sigma)\mapsto l_M(\sigma)$ yields a primitive recursive family of set-sized functions~$l_M$.

Still concerning Definition~\ref{def:search-tree}, we point out that it is not really necessary to repeat the formula $\varphi$ in the case of a conjunction, disjunction or universal quantifier. For example, we could have defined $l_M(\sigma^\frown a)$ as $\Gamma,\varphi_a$ rather than $\Gamma,\varphi,\varphi_a$ in the case of $\varphi\equiv\varphi_0\land\varphi_1$. However, if one omits $\varphi$ in these cases, then the notation becomes more complicated in the case of an existential quantifier: There one would need to consider $l_M(\tau)$ for all initial segments $\tau$ of $\sigma$, rather than just $l_M(\sigma)$ itself. We prefer to repeat~$\varphi$ and keep the simpler notation in the existential case.

As usual, a function $f:\omega\to M$ is called a branch of the tree $S^r_\psi(M)\subseteq M^{<\omega}$ if we have
\begin{equation*}
f\!\restriction\!n:=\langle f(0),\dots,f(n-1)\rangle\in S^r_\psi(M)
\end{equation*}
for every number~$n\in\omega$. The following proposition is typical for the method of deduction chains insofar as a branch of the search tree yields a model. We refer to~\cite[Section~1.3]{freund-thesis} for a detailed formalization of the satisfaction relation in primitive recursive set theory.

\begin{proposition}\label{prop:branch-model}
If $f$ is a branch of $S^r_\psi(M)$, then $M_f:=\{f(i)\,|\,i\in\omega\}\cup\omega\cup\{r\}$ is extensional and we have $(M_f,\in)\vDash\psi(r)$.
\end{proposition}

It is worth observing that the model $M_f$ in the proposition is countable.

\begin{proof}
An $M$-formula is said to occur on $f$ if it is an entry of a sequent $l_M(f\!\restriction\!n)$ for some number~$n$. We will show that every formula that occurs on $f$ is false when relativized to~$M_f$. To deduce the proposition it suffices to observe that $\neg\psi(r)$ and the negation of extensionality occur in $l_M(\langle\rangle)$ and hence on any branch~$f$. The open claim is established by induction over the height of formulas. Let us first consider a literal $\varphi$ that occurs on~$f$. For suitable numbers $k<m$ and $n$ we can write $l_M(f\!\restriction\!n)=\langle\varphi_0,\dots,\varphi_{m-1}\rangle$ with~$\varphi\equiv\varphi_k$. Considering the construction of~$S^r_\psi(M)$, we see that $\varphi$ is the first formula in $l_M(f\!\restriction\!(n+k))$. If $\varphi$ was true, then $f\!\restriction\!(n+k)$ would be a leaf of $S^r_\psi(M)$, again by construction. This would contradict the assumption that~$f$ is a branch. Hence $\varphi$ must be false, as required (relativization to $M_f$ is irrelevant here, since the literal $\varphi$ contains no quantifiers). Let us now consider the case where $\varphi\equiv\forall_x\,\theta(x)$ is universal. As before, we can find an~$n$ such that $\varphi$ is the first formula in $l_M(f\!\restriction\!n)$. Since $f$ is a branch we have
\begin{equation*}
(f\!\restriction\!n)^\frown f(n)=f\!\restriction\!(n+1)\in S^r_\psi(M).
\end{equation*}
By construction, the instance $\theta(f(n))$ occurs in $l_M(f\!\restriction\!(n+1))$ and hence on~$f$.  Since $\theta(f(n))$ is shorter than $\varphi$, the induction hypothesis tells us that $\theta(f(n))$ is false when relativized to $M_f$. In view of $f(n)\in M_f$ it follows that the relativization of $\varphi\equiv\forall_x\,\theta(x)$ to $M_f$ is false as well. Let us now consider the case of an existential formula~$\varphi\equiv\exists_x\,\theta(x)$ that occurs on~$f$. Invoking the induction hypothesis, it suffices to show that each instance $\theta(a)$ with $a\in M_f$ occurs on~$f$ as well. By induction on~$n$ we show that this holds for every~$a$ in the list
\begin{equation}\label{eq:instances}
r,f(0),0,f(1),1,\dots,f(n-1),n-1.
\end{equation}
Assuming the induction hypothesis, let us argue that the formula $\theta(f(n))$ occurs on $f$ (for $a=r$ or $a=n$ one argues similarly): In the construction of $S^r_\psi(M)$ the sequents in the labels are extended and permuted, but no formula is ever removed. Hence we may pick a number~$N>n$ such that $l_M(f\!\restriction\!N)$ contains $\varphi$ and all instances $\theta(a)$ with $a$ in the list~(\ref{eq:instances}). Increasing~$N$ if necessary, we may assume that $\varphi$ is the first formula in~$l_M(f\!\restriction\!N)$. According to the construction of $S^p_\psi(M)$, the instance $\theta(f(n))$ is then added to~$l_M(f\!\restriction\!(N+1))$, unless it was already present. The remaining cases of a conjunction $\varphi\equiv\varphi_0\land\varphi_1$ and a disjunction $\varphi\equiv\varphi_0\lor\varphi_1$ are similar and easier. We do not need to consider the case of a negation, since all our formulas are assumed to be in negation normal form.
\end{proof}

Recall that we write $\sigma\vartriangleleft\tau$ if the sequence $\sigma$ is a proper end-extension of the sequence~$\tau$. Assuming the axiom of dependent choice ($\mathbf{DC}$), we can construct a branch of any tree on which $\vartriangleleft$ is ill-founded. This leads to the following result, in which the first alternative amounts to the conclusion of reflection.

\begin{corollary}[\textbf{DC}]\label{cor:alternative}
Consider a formula $\psi(x)$ with a single free variable $x$ and no parameters. For each real $r\subseteq\omega$ one of the following alternatives must hold:
\begin{enumerate}[label=(\roman*)]
\item We have $(N,\in)\vDash\psi(r)$ for some transitive set $N\ni r$.
\item The order $(S^r_\psi(M),\vartriangleleft)$ is well-founded for every set $M\supseteq\omega\cup\{r\}$.
\end{enumerate}
\end{corollary}
\begin{proof}
Let us show that~(i) holds if~(ii) fails: Assuming the latter, we get a set $M\supseteq\omega\cup\{r\}$ and a non-empty $z\subseteq S^r_\psi(M)$ without a $\vartriangleleft$-minimal element. Consider
\begin{equation*}
T_z:=\{\sigma\in S^r_\psi(M)\,|\,\tau\vartriangleleft\sigma\text{ for some $\tau\in z$}\}
\end{equation*}
and define a binary relation $\vartriangleleft_0$ on $T_z$ by stipulating that $\sigma\vartriangleleft_0\tau$ holds if $\sigma$ is of the form $\sigma=\tau^\frown a$ for some $a\in M$. Using dependent choice, we get a sequence of elements $\sigma_n\in T_z$ with $\sigma_0=\langle\rangle$ and $\sigma_{n+1}\vartriangleleft_0\sigma_n$ for all~$n$. If we define $f(n)$ as the last entry of $\sigma_{n+1}$, then we have $f\!\restriction\!n=\sigma_n$, so that $f$ is a branch of $T_z\subseteq S^r_\psi(M)$. The set $M_f$ from the previous proposition is then a model of $\psi(r)$, and $M_f$ is extensional. Due to the latter, the Mostowski collapse $c:M_f\to N$ with transitive image~$N$ is an $\in$-isomorphism, so that we get $(N,\in)\vDash\psi(c(r))$. Now it suffices to observe that we have $c(r)=r$, since $\omega\cup\{r\}\subseteq M_f$ is transitive. An account of Mostowski collapsing in primitive recursive set theory is given in~\cite[Proposition~1.2.4]{freund-thesis}.
\end{proof}

\section{Class-sized proof trees and truth in the universe}

In the previous section we have constructed search trees $S^r_\psi(M)$ which test whether a given formula $\psi(r)$ holds in some submodel $M_f$ of $M$. We will now show that the trees $S^r_\psi(M)$ glue to a class-sized proof tree $S^r_\psi(\mathbb V)$ with end formula
\begin{equation*}
 \neg\psi(r)\lor\neg\forall_{x,y}(\forall_z(z\in x\leftrightarrow z\in y)\to x=y).
\end{equation*}
If this tree is well-founded, then we can use induction to conclude that $\psi(r)$ fails in the set-theoretic universe~$\mathbb V$. Together with Corollary~\ref{cor:alternative} this will be enough to establish the reflection principle.

Let us begin with a straightforward but crucial observation, which will allow us to glue the trees $S^r_\psi(M)\subseteq M^{<\omega}$ for different arguments~$M$.

\begin{lemma}\label{lem:search-trees-compatible}
Consider sets $M$ and $N$ with $\omega\cup\{r\}\subseteq M\subseteq N$. Then $\sigma\in S^r_\psi(M)$ is equivalent to $\sigma\in S^r_\psi(N)$, for each $\sigma\in M^{<\omega}$. Furthermore, we have $l_M(\sigma)=l_N(\sigma)$ whenever we have $\sigma\in S^r_\psi(M)$.
\end{lemma}
\begin{proof}
The claims can be verified by simultaneous induction over the sequence~$\sigma$, following the recursive clauses from Definition~\ref{def:search-tree}. Crucially, the list
\begin{equation*}
r,\sigma_0,0,\sigma_1,1,\dots,\sigma_{n-1},n-1,n,n+1,n+2,\dots,
\end{equation*}
which is considered if the first formula of $l_M(\sigma)=l_N(\sigma)$ is existential, does only depend on the entries of $\sigma=\langle\sigma_0,\dots,\sigma_{n-1}\rangle$ and not on fixed enumerations of the ambient sets $M$ and $N$.
\end{proof}

The previous lemma shows that $\sigma\in S^r_\psi(M)$ does not depend on $M$, as long as we have $\omega\cup\{r\}\subseteq M$ and $\sigma\in M^{<\omega}$. Given a sequence $\sigma\in\mathbb V^{<\omega}$ with arbitrary entries from $\mathbb V$, an obvious choice for $M$ is
\begin{equation*}
M_\sigma:=\{\sigma_0,\dots,\sigma_{n-1}\}\cup\omega\cup\{r\}\quad\text{for}\quad\sigma=\langle\sigma_0,\dots,\sigma_{n-1}\rangle.
\end{equation*}
The labels of $S^r_\psi(M)$ consist of $M$-formulas, which may contain parameters (constant symbols) from~$M$. In the following we will speak of $\mathbb V$-formulas and $\mathbb V$-sequents to emphasize that arbitrary sets are admitted as parameters. The trees $S^r_\psi(M)$ from Definition~\ref{def:search-tree} can now be glued as follows:

\begin{definition}
Consider a formula $\psi\equiv\psi(x)$ with a single free variable~$x$ and no parameters. For each real $r\subseteq\omega$ we define a class $S^r_\psi(\mathbb V)$ by
\begin{equation*}
\sigma\in S^r_\psi(\mathbb V)\quad:\Leftrightarrow\quad\sigma\in\mathbb V^{<\omega}\text{ and }\sigma\in S^r_\psi(M_\sigma).
\end{equation*}
We also define a class function $l:S^r_\psi(\mathbb V)\to\text{``$\mathbb V$-sequents"}$ by setting $l(\sigma):=l_{M_\sigma}(\sigma)$.
\end{definition}

In the previous section we have observed that $(r,M)\mapsto S^r_\psi(M)$ is a primitive recursive set function. Hence $S^r_\psi(\mathbb V)$ is a primitive recursive class, in the sense that its characteristic function is primitive recursive. It is well-known that primitive recursive set functions have $\Sigma$-definable graphs (see~\cite[Section~2]{jensen-karp}). Thus $S^r_\psi(\mathbb V)$ is $\Delta$-definable, uniformly in the parameter~$r$. Clearly, the relation $\vartriangleleft$ is also primitive recursive and $\Delta$-definable (recall that we have $\sigma\vartriangleleft\tau$ if $\sigma$ is a proper end-extension of $\tau$). The following yields a reformulation of alternative~(ii) from Corollary~\ref{cor:alternative}.

\begin{lemma}\label{lem:wf-class-set}
If $(S^r_\psi(M),\vartriangleleft)$ is well-founded for every set $M\supseteq\omega\cup\{r\}$, then the class $(S^r_\psi(\mathbb V),\vartriangleleft)$ is well-founded as well.
\end{lemma}
\begin{proof}
Given a non-empty $z\subseteq S^r_\psi(\mathbb V)\subseteq\mathbb V^{<\omega}$, we put
\begin{equation*}
M:=\bigcup\{M_\sigma\,|\,\sigma\in z\}.
\end{equation*}
Due to Lemma~\ref{lem:search-trees-compatible} we have $z\subseteq S^r_\psi(M)$. Now the assumption of the lemma yields a $\vartriangleleft$-minimal element of~$z$, as required.
\end{proof}

In the proof of our main result we will want to use induction over $\sigma\in S^r_\psi(\mathbb V)$ to show that each sequent $l(\sigma)$ contains a true formula. The following observation ensures that the required truth definition is available.

\begin{lemma}\label{lem:subformula-property}
Any formula in a sequent $l(\sigma)$ with $\sigma\in S^r_\psi(\mathbb V)$ is a substitution instance of a subformula of $\neg\psi(r)$ or of $\neg\forall_{x,y}(\forall_z(z\in x\leftrightarrow z\in y)\to x=y)$.
\end{lemma}
\begin{proof}
In view of the definition of $S^r_\psi(\mathbb V)$, it suffices to establish the claim for all formulas that occur in some sequent $l_M(\sigma)$ with $M\supseteq\omega\cup\{r\}$ and~$\sigma\in S^r_\psi(M)$. For fixed~$M$, this can be accomplished by a straightforward induction over the sequence~$\sigma$, which follows the recursive clauses from Definition~\ref{def:search-tree}.
\end{proof}

We now have all ingredients to complete the proof of our main result. In the introduction we have given precise formulations of the reflection and induction principles that it involves.

\begin{theorem}\label{thm:main-result}
The following are equivalent over primitive recursive set theory extended by the axiom of dependent choice:
\begin{enumerate}[label=(\arabic*)]
\item reflection for first order formulas with real parameters,
\item induction over $\Delta(\mathbb R)$-definable well-founded classes. 
\end{enumerate}
\end{theorem}
\begin{proof}
From Proposition~\ref{prop:reflection-to-induction} we already know that~(1) implies~(2). For the converse direction we assume~(2) and establish an arbitrary instance
\begin{equation*}
\forall_{r\subseteq\omega}(\psi(r)\to\exists_N(\text{``$N\ni r$ is transitive"}\land (N,\in)\vDash\psi(r)))
\end{equation*}
of~(1). Aiming at the contrapositive, consider a real $r$ such that $\psi(r)$ does not hold in any transitive model $N\ni r$. By Corollary~\ref{cor:alternative} and Lemma~\ref{lem:wf-class-set}, it follows that $(S^r_\psi(\mathbb V),\vartriangleleft)$ is a well-founded class. We will use induction over the latter to show that each sequent $l(\sigma)$ with $\sigma\in S^r_\psi(\mathbb V)$ contains a true formula. Note that the reference to truth is unproblematic, as Lemma~\ref{lem:subformula-property} ensures that we are only concerned with instances of finitely many formulas. If the induction is successful, it tells us that there is a true formula in the sequent
\begin{equation*}
l(\langle\rangle)=\langle\neg\psi(r),\neg\forall_{x,y}(\forall_z(z\in x\leftrightarrow z\in y)\to x=y)\rangle.
\end{equation*}
Since extensionality is an axiom of primitive recursive set theory, this means that the formula $\neg\psi(r)$ must be true in the set-theoretic universe. Hence reflection for $\psi(r)$ holds because its premise fails. It remains to carry out the induction. To prove the induction step, we distinguish cases according to the first formula in
\begin{equation*}
l(\sigma)=\varphi,\Gamma.
\end{equation*}
Note that the sequent $l(\sigma)$ is never empty, as pointed out in Definition~\ref{def:search-tree}. Let us begin with the case where $\varphi$ is a literal. If $\varphi$ is true, then we are done. Otherwise, Definition~\ref{def:search-tree} yields $\sigma^\frown 0\in S^r_\psi(M_\sigma)$ and $l_{M_\sigma}(\sigma^\frown 0)=\Gamma,\varphi$. In view of $M_{\sigma^\frown 0}=M_\sigma$ we get $\sigma^\frown 0\in S^r_\psi(\mathbb V)$ and $l(\sigma^\frown 0)=\Gamma,\varphi$. We also have $\sigma^\frown 0\vartriangleleft\sigma$, so that the induction hypothesis yields a true formula in~$l(\sigma^\frown 0)$. In the present case, the sequent $l(\sigma^\frown 0)$ is just a permutation of $l(\sigma)$, so that the latter contains the same true formula. Next, we consider the case of an existential formula $\varphi\equiv\exists_x\,\theta(x)$. Similarly to the previous case, the construction from Definition~\ref{def:search-tree} leads to
\begin{equation*}
\sigma^\frown 0\in S^r_\psi(\mathbb V)\quad\text{and}\quad l(\sigma^\frown 0)=\Gamma,\varphi,\theta(b),
\end{equation*}
for a certain parameter~$b$. Again, the induction hypothesis yields a true formula in the sequent $l(\sigma^\frown 0)$. If this formula lies in $\Gamma,\varphi$, then it occurs in $l(\sigma)$ and we are done. So now assume that $\theta(b)$ is true. Then $b$ witnesses the truth of the existential formula~$\varphi$, and we are done as well. In the most interesting case we are concerned with a universal formula $\varphi\equiv\forall_x\,\theta(x)$. Aiming at a contradiction, we assume that every formula in $l(\sigma)$ is false. In particular $\varphi$ is false, and we may pick an $a\in\mathbb V$ such that $\theta(a)$ is false as well. In view of $M_\sigma\subseteq M_{\sigma^\frown a}$ we can use Lemma~\ref{lem:search-trees-compatible} to obtain $\sigma\in S^r_\psi(M_{\sigma^\frown a})$ and
\begin{equation*}
l_{M_{\sigma^\frown a}}(\sigma)=l_{M_\sigma}(\sigma)=l(\sigma)=\varphi,\Gamma.
\end{equation*}
By definition we have $a\in M_{\sigma^\frown a}$, so that Definition~\ref{def:search-tree} yields $\sigma^\frown a\in S^r_\psi(M_{\sigma^\frown a})$ and hence $\sigma^\frown a\in S^r_\psi(\mathbb V)$. We also get
\begin{equation*}
l(\sigma^\frown a)=l_{M_{\sigma^\frown a}}(\sigma^\frown a)=\Gamma,\varphi,\theta(a).
\end{equation*}
In view of $\sigma^\frown a\vartriangleleft\sigma$, the induction hypothesis tells us that some formula in this sequent is true. This contradicts the assumption that all formulas in $l(\sigma)=\varphi,\Gamma$ and the instance $\theta(a)$ are false. The remaining cases of a disjunction $\varphi\equiv\varphi_0\lor\varphi_1$ and of a conjunction $\varphi\equiv\varphi_0\land\varphi_1$ are similar and easier.
\end{proof}

Let us conclude this paper with some observations on formula complexity. Concerning induction along $\Delta(\mathbb R)$-definable well-founded classes, we first show that \mbox{$\Sigma_n$-induction} is equivalent to $\Pi_{n+1}$-induction, for $n>0$. The following argument is similar to a proof by R.~Sommer~\cite[Lemma~4.5]{sommer95}: We want to establish induction for a $\Pi_{n+1}$-formula $\varphi(x)\equiv\forall_y\,\theta(x,y)$ along a $\Delta(\mathbb R)$-definable class $(X,<_X)$. For this purpose we order the class $X\times\mathbb V$ of pairs by
\begin{equation*}
\langle x,y\rangle<_{X\times\mathbb V}\langle x',y'\rangle\quad:\Leftrightarrow\quad x<_X x'.
\end{equation*}
It is straightforward to show that $<_{X\times\mathbb V}$ is well-founded if the same holds for $<_X$. Assume that $\varphi$ satisfies the premise of induction, and consider the $\Sigma_n$-formula
\begin{equation*}
\theta'(z):\equiv\exists_{x,y}(z=\langle x,y\rangle\land\theta(x,y)).
\end{equation*}
To establish the induction step for $\theta'$, we need to deduce $\theta(x,y)$ from the assumption that $\theta(x',y')$ holds for all pairs $\langle x',y'\rangle<_{X\times\mathbb V}\langle x,y\rangle$. The latter amounts to
\begin{equation*}
\forall_{x'<_X x}\forall_{y'}\theta(x',y')\equiv\forall_{x'<_X x}\,\varphi(x').
\end{equation*}
Since $\varphi$ satisfies the premise of induction, we get $\varphi(x)$ and in particular $\theta(x,y)$ for the relevant~$y$. Now $\Sigma_n$-induction along $<_{X\times\mathbb V}$ yields
\begin{equation*}
\forall_{x\in X}\forall_y\theta(x,y)\equiv\forall_{x\in X}\varphi(x).
\end{equation*}
This is the conclusion of induction for the $\Pi_{n+1}$-formula $\varphi$.

We now discuss the amount of reflection that is needed to deduce induction for~a given formula. In view of the previous paragraph, we may focus on an induction formula $\varphi$ of complexity $\Pi_{n+1}$. In this case the reflection formula
\begin{equation*}
  \psi(r)\equiv\forall_{x,y}(\theta_\Sigma(x,y,r)\leftrightarrow\theta_\Pi(x,y,r))\land\neg\ind[X^r,\varphi]
 \end{equation*}
 from the proof of Proposition~\ref{prop:reflection-to-induction} has complexity $\Sigma_{n+3}$ (over Kripke-Platek set theory, where $\Sigma$-formulas and $\Sigma_1$-formulas coincide). In view of J\"ager and Strahm's analysis of $\omega$-model reflection~\cite{jaeger-strahm-bi-reflection}, one might have expected a slightly stronger result, which would deduce $\Pi_{n+1}$-induction from $\Pi_{n+2}$-reflection (the latter does not seem to imply $\Sigma_{n+3}$-reflection in our setting, due to the restriction to real parameters). To prove this strengthening one could try to replace $\ind[X^r,\varphi]$ by
 \begin{equation*}
 \forall_{x\in X^r}(\forall_{y<_X^r x}\varphi(y,\vec c)\to\varphi(x,\vec c))\to\forall_{x\in X^r}\varphi(x,\vec c),
 \end{equation*}
where the quantified variables $\vec z$ from $\ind[X^r,\varphi]$ are instantiated to suitable parameters~$\vec c$. If $\varphi$ has complexity $\Pi_{n+1}$, then the displayed formula has complexity~$\Sigma_{n+2}$, and the reflection formula has complexity~$\Pi_{n+2}$. The problem is that we cannot apply our reflection principle to the new formula, since the parameters $\vec c$ may not be reals. Our observation does show that reflection for $\Pi_{n+2}$-formulas implies induction for $\Pi_{n+1}$-formulas with real parameters.

Finally, we consider the passage from induction to reflection: In the proof of Theorem~\ref{thm:main-result} we have used induction over $\sigma\in S^r_\psi(\mathbb V)$ to show that the sequent~$l(\sigma)$ contains a true formula. If $\psi$ has complexity $\Sigma_{n+1}$ with $n>0$, then $l(\sigma)$ can only contain $\Pi_{n+1}$-formulas, due to Lemma~\ref{lem:subformula-property}. If $\psi$ is a $\Pi_{n+2}$-formula, we can lower the complexity by assuming that $\psi(r)$ holds (since this is the premise of reflection). Under this assumption, any true formula in $l(\sigma)$ must again have complexity $\Pi_{n+1}$. Hence $\Pi_{n+1}$-induction is sufficient in both cases. Altogether, we have seen that $\Sigma_{n+3}$-reflection implies $\Pi_{n+1}$-induction (or equivalently $\Sigma_n$-induction), which in turn implies $\Pi_{n+2}$-reflection, for $n>0$. The slight mismatch in terms of logical complexity seems to be caused by the partial restriction to real parameters.

\bibliographystyle{amsplain}
\bibliography{Set-reflection-induction}

\end{document}